\documentclass[a4paper,12pt,onecolumn]{article}

\usepackage{hyperref}
\usepackage[vmargin=2cm,hmargin=2cm]{geometry}
\usepackage{bm}
\usepackage[utf8]{inputenc}
\usepackage{empheq}
\usepackage{stackrel}
\usepackage{cases}
\usepackage{mathtools}
\usepackage{amsthm,amsmath,amscd}
\usepackage{makeidx}
\usepackage[charter]{mathdesign}
\usepackage{perpage}
\usepackage{tikz-cd}
\tikzcdset{every label/.append style = {font = \small}}
\usepackage{caption}
\usepackage{graphicx}
\DeclareMathSizes{12}{12}{8}{6}

\usepackage{cite}
\usepackage{url}
\usepackage{accents}

\newtheoremstyle{ptheorem}{1em}{0em}{\itshape}{}{\bfseries}{.}{.5em}{\thmname{#1}\thmnumber{ #2}\thmnote{ (\hspace{-.01pt}{#3})}}

\theoremstyle{ptheorem}

\newtheorem{thm}{Theorem}[section]
\newtheorem{pro}[thm]{Proposition}

\newtheorem{cor}[thm]{Corollary}

\newtheoremstyle{hdef}{1em}{0em}{}{}{\bfseries}{.}{.5em}{\thmname{#1}\thmnumber{ #2}\thmnote{ (\hspace{-.01pt}{#3})}}
\theoremstyle{hdef}

\newtheorem{dfn}[thm]{Definition}
\newtheorem{rem}[thm]{Remark}
\newtheorem{notation}[thm]{Notation}
\makeatletter
\newtheoremstyle{premark}{1em}{0em}{
\addtolength{\@totalleftmargin}{1.5em}
\addtolength{\linewidth}{-1.5em}
\parshape 1 1.5em \linewidth}{}{\scshape}{.}{.5em}{}
\makeatother

\theoremstyle{hdef}

\newtheorem{exa}[thm]{Example}

\numberwithin{equation}{section}
\numberwithin{figure}{section}

\allowdisplaybreaks

\author{
F. Adri\'an F. Tojo\footnote{The author was partially supported by Ministerio de Econom\'ia y Competitividad, Spain, and FEDER, project MTM2013-43014-P, and by the Agencia Estatal de Investigaci\'on (AEI) of Spain under grant MTM2016-75140-P, co-financed by the European Community fund FEDER.} \\
\normalsize e-mail: fernandoadrian.fernandez@usc.es\\
\normalsize \emph{Instituto de Ma\-te\-m\'a\-ti\-cas, Facultade de Matem\'aticas,} \\ \normalsize\emph{Universidade de Santiago de Com\-pos\-te\-la, Spain.}\\ 
}

\begin{document}
\title{Green's Functions of Recurrence Relations\\ with Reflection}

\date{}

\maketitle


\begin{abstract}
In this work we develop an algebraic theory of linear recurrence equations and systems with constant coefficients and reflection. We obtain explicit solutions and the Green's functions associated to different problems under general linear boundary conditions. Furthermore, we establish different relations between the algebras of recurrence and differential operators, showing the similarities and differences between them.
\end{abstract}

{\small\textbf{Keywords:} Reflection, Recurrence Relations, System of First Order Difference Equations, General Boundary Conditions

}

\section{Introduction}

In recent years, the study of differential equations with reflection has progressed through various research lines. On one hand we have those works that deal with qualitative applications, such as boundedness \cite{Aft}, periodicity \cite{Cabada2014b} or existence and uniqueness of solution \cite{Cab5,Gup,Ore}. Other articles find Hilbert bases through operator eigenfunction decomposition \cite{Sarsenbi1,Sarsenbi2}. Finally, we have those works in which the authors obtain explicit solutions or the associated Green's functions. That is the case of \cite{CabToj,Toj3,Cab4,CabToj2} and specially of \cite{CTMal,CaTo}, where they develop a general theory of Green's functions in the case of differential equations and differential systems respectively.

Despite all of this progress in the field, there have not been any works yet in which the authors obtain Green's functions of recurrence relations with reflection, something that, following the usual parallelism between differential and difference equations, should be possible. The aim of this work is therefore to fill this void in the theory, by providing and algebraic theory of recurrence relations and systems with reflection and constructing the Green's functions associated to different problems.

The basic idea exploited in \cite{CTMal,CaTo} is to endow differential equations with reflection with an adequate algebraic structure. In order to achieve this, the authors first observe that homogeneous linear differential equations with reflection and constant coefficients can always be expressed in the form 
\begin{equation}\label{eqdif}Tu(t):=\sum_{k=0}^na_ku^{(k)}(t)+\sum_{k=0}^nb_ku^{(k)}(-t)=0.\end{equation}

The operator $T$ in~\eqref{eqdif} can be considered as a composition of simpler operators. First, we have the usual \emph{differential operator} which we will note by $\widetilde D$, but also we have to consider the \emph{pullback by the reflection} function $\varphi(t)=-t$, that is, the operator $\varphi^*$ such that $(\varphi^* f)(t)=f(-t)$ for any function $f:{\mathbb R}\to{\mathbb R}$.

Now we can consider the algebra of linear differential operators with reflection ${\mathbb R}[\widetilde D,\widetilde\varphi^*]$ as defined in \cite{CaTo}. This algebra consists of all operators of the form of $T$. These operators can be written as $\widetilde\varphi^*P(\widetilde D)+Q(\widetilde D)$ where $P$ and $Q$ belong to ${\mathbb R}[\widetilde D]$, that is, the real polynomials on the abstract variable $\widetilde D$. The algebraic structure is provided by the usual composition of operators and the rules derived from it. For instance, $(\widetilde\varphi^*)^2=\operatorname{Id}$, where $\operatorname{Id}$ is the \emph{identity operator}, and, if we write $\varphi^*(P)(\widetilde D):=P(-\widetilde D)$, we have that $ P\circ\widetilde \varphi^*=\widetilde \varphi^*\circ\widetilde \varphi^*(P)$.

In the case of the operator $T$ in~\eqref{eqdif} it can be expressed as
\begin{equation}\label{Lop}T=\sum_ka_k\widetilde\varphi^*\widetilde D^k+\sum_kb_k\widetilde D^k\in{\mathbb R}[\widetilde D,\varphi^* ].\end{equation}
In \cite{CTMal} we find results that allow us to obtain the solution of differential problems with such operators.
\begin{thm}[{\cite[Theorem 2.1]{CTMal}}]\label{thmdec}
Take $T$ defined as in~\eqref{Lop} and take
\begin{equation}\label{Rop}R=\sum_{k}a_k\varphi^*\widetilde D^k+\sum_{l}(-1)^{l+1}b_l\widetilde D^l\in{\mathbb R}[\widetilde D,\varphi^* ].\end{equation} Then $RT=TR\in{\mathbb R}[\widetilde D]$.
\end{thm}
\begin{thm}[{\cite[Theorem 3.2]{CTMal}}]\label{thmdei} Consider the problem
\begin{equation}\label{rbvp}Tu(t)=h(t),\ t\in [-T,T],\ B_iu=0,\ i=1,\dots,n,
\end{equation}
where $T$ is defined as in~\eqref{Lop}, $h\in L^1([-T,T])$ and
\begin{displaymath}B_iu:=\sum_{j=0}^{n-1}\alpha_{ij}u^{(j)}(-T)+\beta_{ij}u^{(j)}(T).\end{displaymath}
Then, there exists $R\in {\mathbb R}[\widetilde D,\varphi^* ]$ (as in~\eqref{Rop}) such that $S:=RT\in{\mathbb R}[\widetilde D]$ and the unique solution of problem~\eqref{rbvp} is given by $\int_a^bR_\vdash G(t,s)h(s)\operatorname{d} s$ where $G$ is the Green's function associated to the problem $Su=0$, $B_iRu=0$, $B_iu=0$, $i=1,\dots,n$, assuming that it has a unique solution.
\end{thm}

An analogous study can be done for linear systems with reflection with the same algebraic structure --see \cite{CaTo}. Take, for instance, the system
\begin{equation}\label{hlsystemv}Hu(t):=Fu'(t)+Gu'(-t)+A u(t)+Bu(-t)=0, t\in{\mathbb R}.
\end{equation}

In this context we find the following results.

\begin{thm}[{\cite[Theorem 4.5]{CaTo}}]\label{thmexpfmv}
Assume $F-G$ and $F+G$ are invertible. Then
\begin{equation}\label{Xseries}X(t): = \sum_{k=0}^\infty\frac{E^k t^{2k}}{(2k)!} -(F+G)^{-1}(A+B)\sum_{k=0}^\infty\frac{E^k t^{2k+1}}{(2k+1)!},\end{equation}
where $E=(F-G)^{-1}(A-B)(F+G)^{-1}(A+B)$, is a fundamental matrix of problem~\eqref{hlsystemv}. If we further assume $A-B$ and $A+B$ are invertible, then $E$ is invertible and we can consider a square root $\Omega$ of $E$. Then,
\begin{equation}\label{fme}X(t)=\cosh \Omega t -(F+G)^{-1}(A+B)\Omega^{-1}\sinh\Omega t.\end{equation}
\end{thm}
Consider now the initial value problem
\begin{align}\label{pgf1}Fu'(t)+Gu'(-t)+A u(t)+Bu(-t) & =\gamma,\ t\in {\mathbb R},\\ \label{pgf2} u(0) & =\delta,
\end{align}
where $A,B,F,G\in{\mathcal M}_n({\mathbb R})$, $\gamma\in{\mathcal C}({\mathbb R})$, and $\delta\in{\mathbb R}^n$. 
\begin{thm}[{\cite[Theorem 6.1]{CaTo}}]Consider the problems
\begin{equation}\label{hlsystemnh}Fu'(t)+Gu'(-t)+A u(t)+Bu(-t)=\gamma, t\in{\mathbb R},
\end{equation}
and
\begin{equation}\label{hlsystemnh2}Fu'(t)-Gu'(-t)+A u(t)-Bu(-t)=\gamma, t\in{\mathbb R}.
\end{equation} Assume $F+G$ and $F-G$ are invertible, $X$ and $Y$ are fundamental matrices of problems~\eqref{hlsystemnh} and~\eqref{hlsystemnh2} respectively and $\bm{\mathcal X}$ is invertible in ${\mathbb R}$. Then problem~\eqref{pgf1}--\eqref{pgf2} has a unique solution $u:{\mathbb R}\to{\mathbb R}^n$ and it is given by \[ u(t)=X(t)X(0)^{-1}\delta+\int_{-t}^tG(t,s)\gamma(s)\operatorname{d} s,\]  where
\[ G(t,s)=\begin{dcases} \frac{1}{2}\left( \begin{array}{c|c} X(t) & Y(t)\end{array}\right) {\bm{\mathcal X}(s)}^{-1}\begin{pmatrix} (F-G)^{-1} \\ \hline (F+G)^{-1}\end{pmatrix}, & 0\leqslant s\leqslant t,\\
\frac{1}{2}\left( \begin{array}{c|c} X(t) & Y(t)\end{array}\right) {\bm{\mathcal X}(-s)}^{-1}\begin{pmatrix} -(F-G)^{-1} \\ \hline (F+G)^{-1}\end{pmatrix}, & -t\leqslant s<0.
\end{dcases}\] 
\end{thm}

Our objective will be to obtain similar results as the ones presented above for the case of linear recurrence equations and systems with reflection. In this work we will build a similar algebraic structure for the case of recurrence relations, pinpointing the similarities and differences with the algebra ${\mathbb R}[\widetilde D,\widetilde\varphi^*]$. In Section 2 we define the algebra ${\mathbb F}[D,D^{-1},\varphi^*]$ of recurrence relations with reflection and study its properties as well as its relation to the algebra ${\mathbb F}[\widetilde D,\widetilde\varphi^*]$. In Section 3 we provide Green's functions for recurrence relations with reflection and general boundary conditions and in Section 4 we provide an analogous theory for linear systems with reflection. Finally, in Section 5 we establish the conclusions regarding the theory and pose several open problems worth studying.

\section{Recurrence relations with reflection}

Let us first set up the basic definitions and notation in order to study recurrence relations with reflection in the highest generality.

\subsection{Definitions and notation}
Given two sets $A$ and $B$ we denote by ${\mathcal F}(A,B)$ the space of functions $f: A\to B$. Let ${\mathbb F}$ be a field, $\overline{\mathbb F}$ its algebraic closure and $V$ a vector space over ${\mathbb F}$. Let ${\mathscr S}$ be the space of ${\mathbb Z}$-sequences in $V$ that is, ${\mathscr S}:={\mathcal F}({\mathbb Z},V)$. ${\mathscr S}$ is an ${\mathbb F}$-vector space. Given $x\in{\mathscr S}$ we write $x(k)\equiv x_k\equiv(x)_k$ and $x\equiv(x_k)_{k\in{\mathbb Z}}$. We define the \emph{right shift operator} $D$ as
\begin{center}\begin{tikzcd}[row sep=tiny]
{\mathscr S} \arrow{r}{D} & {\mathscr S}\\
(x_k)_{k\in{\mathbb Z}} \arrow[mapsto]{r} & (x_{k+1})_{k\in{\mathbb Z}}
\end{tikzcd}
\end{center}
$D$ is bijective and, in the present discussion, it will play the role the differential operator does in differential equations (thence the $D$ as notation). That role could also be played by the forward difference operator $\Delta:=D-\operatorname{Id}$ but, for simplicity, we stick to $D$. \par
An \emph{order $n$ linear recurrence relation} (sometimes referred as \emph{difference equation}, although there is a subtle difference between the two of them \cite{Mill}) with constant coefficients is normally expressed as
\begin{equation}\label{recrel}x_{k+n}=\sum_{j=0}^{n-1}a_{j}x_{k+j}+c_k,\ k\in{\mathbb N};\quad x_k=\xi_k,\ k=1,\dots,n,\end{equation}
where $\xi_k\in{\mathbb F},\ k=1,\dots,n$; $a_j\in{\mathbb F},\ j=0,\dots,n-1$; $a_0\ne 0$ and $c=(c_k)_{k\in{\mathbb N}}$. A \emph{solution} of the difference equation~\eqref{recrel} will be a sequence $u=(u_k)_{k\in{\mathbb N}}$ such that equation~\eqref{recrel} holds when substituting $x_k$ by $u_k$ for every $k\in{\mathbb N}$.

Using operator $D$, we can rewrite the recurrence relation~\eqref{recrel} as
\begin{align*}\left(D^n-\sum_{j=0}^{n-1}a_jD^j\right)x=c;\quad x_k=\xi_k,\ k=1,\dots,n,\end{align*}
where $x=(x_k)_{k\in{\mathbb N}}$. So, it is only fitting that we study equations of the kind
\begin{equation}\label{eqdifor}Ux:=\sum_{j=0}^na_jD^jx=c;\quad x_k=\xi_k,\ k=1,\dots,n,\end{equation}
where $a_0a_n\ne 0$. We say that $U$ occurring in~\eqref{eqdifor} belongs to ${\mathbb F}[D]$, the algebra of polynomials on $D$ with coefficients in ${\mathbb F}$. \par Now we introduce reflections in this context, which forces us to work on ${\mathbb Z}$ instead of ${\mathbb N}$. Let $\varphi:{\mathbb Z}\to{\mathbb Z}$ be such that $\varphi(t)=-t$. We define the pullback by $\varphi$, $\varphi^*$, as
\begin{center}\begin{tikzcd}[row sep=tiny]
{\mathscr S} \arrow{r}{\varphi^*} & {\mathscr S}\\
(x_k)_{k\in{\mathbb Z}} \arrow[mapsto]{r} &(x_{\varphi(k)})_{k\in{\mathbb Z}}
\end{tikzcd}
\end{center}

We can consider now linear difference equations with reflection of the form
\begin{equation}\label{difeqin}Lx:=\sum_{j=-n}^n\left( a_j+b_j\varphi^*\right) D^jx=c,\end{equation}
where $x,c\in{\mathscr S}$; $a_{j},b_j\in {\mathbb F}$ for $j=0,\dots,n$ and $D^{-j}=(D^{-1})^j$ for $j\in{\mathbb N}$. We say $L$ belongs to the operator algebra ${\mathbb F}[D,D^{-1},\varphi^*]$ generated by $D^j$ and $\varphi^*D^j$, $j\in{\mathbb Z}$ with the composition operation. We will omit the composition sign while working in this algebra.\par

\subsection{Algebraic structure}

In this section we enter the algebraic structure of ${\mathbb F}[D,D^{-1},\varphi^*]$ in greater detail.

\begin{dfn}An expression of the kind $\sum_{j\in{\mathbb Z}}a_jD^j$ where $a_j\in{\mathbb F}$ and only finitely many elements of $\{a_j\}_{j\in{\mathbb Z}}$ are nonzero is called a \emph{formal Laurent polynomial} on the variable $D$. We will denote the set of Laurent polynomials in the variable $D$ by ${\mathbb F}[D,D^{-1}]$. This set has a natural structure of commutative ${\mathbb F}$-algebra with the sum, product by scalars and composition of operators --which is the product of Laurent polynomials in this case.
\end{dfn}
\begin{rem}Other realizations of the algebra ${\mathbb F}[D,D^{-1}]$ can be achieved. For instance, it can be considered as the algebra of (commutative) polynomials in two variables ${\mathbb F}[D,E]$ quotiented by the relation $ED=\operatorname{Id}$.

 Similarly, the operator algebra ${\mathbb F}[D,D^{-1},\varphi^*]$ is the quotient of the algebra of non commutative polynomials ${\mathbb F}\left< D,E,F\right>$ by the relations $DF=FE$, $DE=1$ and $F^2=1$.
Observe that a basic property of the interaction between $D$ and $\varphi^*$ is that
$D\varphi^*=\varphi^*D^{-1}$. In fact, we have that $P\varphi^*=\varphi^*\varphi^*(P)$ where $\varphi^*(P)(D):=P(D^{-1})$ for any $P\in{\mathbb F}[D,D^{-1}]$, that is,
${\mathbb F}[D,D^{-1},\varphi^*]$ consists of the operators of the form $\varphi^*P+Q$ with $P,\,Q\in{\mathbb F}[D,D^{-1}]$. It is for this reason that the operators defining linear recurrence relations with reflection can be reduced to those occurring in usual ordinary difference equations, as the following theorem shows.

\begin{thm}\label{T::RL1} Let $L=\varphi^*P+Q$ with $P,\,Q\in{\mathbb F}[D,D^{-1}]$. Then $R:=\varphi^*P-\varphi^*(Q)\in{\mathbb F}[D,\varphi^* ]$ satisfies $RL=LR\in{\mathbb F}[D,D^{-1}]$.
\end{thm}
\begin{proof}
 \begin{align*}RL= & (\varphi^*P-\varphi^*(Q))(\varphi^*P+Q)=\varphi^*P\varphi^*P-\varphi^*(Q)Q+\varphi^*PQ-\varphi^*(Q)\varphi^*P \\= & \varphi^*(P)P-\varphi^*(Q)Q+\varphi^*PQ-\varphi^*QP=\varphi^*(P)P-\varphi^*(Q)Q\in{\mathbb F}[D,D^{-1}].\end{align*}

The same holds for $LR$.

\end{proof}

\begin{rem} Observe that, if $L$ is of the form in~\eqref{difeqin}, we have that $LRD^{2n}\in{\mathbb F}[D]$, but the same may hold for exponents $k<2n$. We will assume from now on that we take the least of these exponents. Also, in the particular case $a_j,b_j=0$ for $j<0$, we have that $LRD^{n}\in{\mathbb F}[D]$.
\end{rem}

Now, observe that, for any $P\in{\mathbb F}[D,D^{-1}]\backslash\{0\}$, $P$ can be expressed uniquely as $P(D)=P_*(D)D^k$ for some $P_*\in{\mathbb F}[D]$ without zero as root and $k\in{\mathbb Z}$. If $P,Q\in{\mathbb F}[D,D^{-1}]$, we say that $P$ \emph{divides} $Q$, and we write $P|Q$, if $P_*$ divides $Q_*$.

We can consider the set ${\mathbb F}_*[D]$ of ${\mathbb F}$-polynomials on the variable $D$ without zero roots --which is isomorphic to ${\mathbb F}[D]/(D)$, that is, ${\mathbb F}[D]$ quotiented by the ideal generated by $D$. Take also the product ${\mathbb F}_*[D]\times{\mathbb Z}$ --or, which is the same $({\mathbb F}[D]/(D))\times (D)$-- and the bijection
\begin{center}\begin{tikzcd}[row sep=tiny]
{\mathbb F}[D,D^{-1}] \arrow{r}{\Psi} & {\mathbb F}_*[D]\times{\mathbb Z}\\
P_*(D)D^k \arrow[mapsto]{r} &(P_*,k)
\end{tikzcd}
\end{center}
with inverse
\begin{center}\begin{tikzcd}[row sep=tiny]
{\mathbb F}_*[D]\times{\mathbb Z} \arrow{r}{\Psi^{-1}} & {\mathbb F}[D,D^{-1}]\\
(Q,k) \arrow[mapsto]{r} &Q(D)D^k
\end{tikzcd}
\end{center}
Inducing the algebra operations of ${\mathbb F}[D,D^{-1}]$ in ${\mathbb F}_*[D]\times{\mathbb Z}$ we get, for $(P,k),(Q,j)\in {\mathbb F}_*[D]\times{\mathbb Z}$ and $\lambda\in{\mathbb F}$,
\begin{align*}
\lambda\cdot(P,k) & =(\lambda P,k), \\
(P,k)\cdot(Q,j) & =(P\,Q,k+j),\\
(P,k)+(Q,j) & =\Psi(P(D)D^k+Q(D)D^j).
\end{align*}
\end{rem}

We can express the relation $P(D)\varphi^*=\varphi^*P(D^{-1})$ in terms of ${\mathbb F}_*[D]\times{\mathbb Z}$ in the following way:
\[ \left( \alpha\prod_{j=1}^n(x-\lambda_j),k\right) \varphi^*=\varphi^*\left( (-1)^n\alpha\prod_{j=1}^n\lambda_j\prod_{j=1}^n\left( x-\frac{1}{\lambda_j}\right) ,-k-n\right) ,\] 
for any $\left( \alpha\prod_{j=1}^n(x-\lambda_j),k\right) \in{\mathbb F}_*[D]\times{\mathbb Z}$.\par 

The algebra isomorphism $\Psi$ allows us to define greatest common divisor ($\operatorname{gcd}$) in ${\mathbb F}[D,D^{-1}]$ through ${\mathbb F}_*[D]\times{\mathbb Z}$. Remember that the greatest common divisor of $P,Q\in{\mathbb F}[D]$ is the product of the monomials $D-\lambda\operatorname{Id}$ where $\lambda\in\overline {\mathbb F}$ is a common eigenvalue of $P$ and $Q$.

\begin{dfn} We define the \emph{greatest common divisor} of $(P_1,k_1),\dots,(P_n,k_n)\in{\mathbb F}_*[D]\times{\mathbb Z}$ as
\[ \operatorname{gcd}\{(P_1,k_1),\dots,(P_n,k_n)\}=(\operatorname{gcd}\{P_1,\dots,P_n\},\nu(k_1,\dots,k_n)),\] 
where 
\[ \nu(k_1,\dots,k_n)=\begin{dcases} \min\{k_1,\dots,k_n\}, & k_j\geqslant 0;\ j=1,\dots,n,\\\max\{k_1,\dots,k_n\}, & k_j\leqslant 0;\ j=1,\dots,n,\\ 0, & \text{otherwise.}\end{dcases}\] 
\end{dfn}

For $L=\varphi^*P+Q\in{\mathbb F}[D,D^{-1},\varphi^*]$ with $P,Q\in{\mathbb F}[D,D^{-1}]$ let \[ \overline L:=\operatorname{gcd}(P,\varphi^*(Q)).\] 
By construction, $\overline L|P$ and $\overline L|\varphi^*(Q)$. Let $\widetilde P=P/\overline L$ and $\widetilde Q=\varphi^*(Q)/\overline L$.

Using the above expressions and the algebraic structure, we can improve Theorem~\ref{T::RL1} in the following way --cf. \cite[Theorem 2.3]{CaTo}.

\begin{thm}\label{T::RL2}
	Take $L$, $\widetilde P$ and $\widetilde Q$ as above and define
	$\widetilde R:=\varphi^*\widetilde P-\widetilde Q\in{\mathbb F}[D,D^{-1},\varphi^*]$. Then $L\widetilde R\in{\mathbb F}[D,D^{-1}]$.
\end{thm}
\begin{proof}
 \begin{align*}L\widetilde R= & (\varphi^* P+ Q)(\varphi^*\widetilde P-\widetilde Q)=\varphi^*\widetilde P\varphi^* P-Q\widetilde Q+Q\varphi^*\widetilde P-\varphi^* P\widetilde Q \\= & \varphi^*(\widetilde P) P-Q\widetilde Q+ \varphi^*\varphi^*(Q)\widetilde P-\varphi^*\overline L\widetilde P\widetilde Q=\varphi^*(\widetilde P) P-Q\widetilde Q+ \varphi^*[\overline L\widetilde Q\widetilde P-\overline L\widetilde P\widetilde Q]=\varphi^*(\widetilde P) P-Q\widetilde Q.\end{align*}
\end{proof}
\begin{rem} Unlike Theorem~\ref{T::RL1}, we do not have in Theorem~\ref{T::RL2} that $L\widetilde R=R\widetilde L$, but this commutativity is not in general necessary.
\end{rem}
\begin{rem}\label{remred}
	From previous Theorem, it is clear that, as in Theorem~\ref{T::RL1}, there exists a least $k\in\{0,1,2,\dots\}$ such that $ L\widetilde RD^k\in {\mathbb F} [D]$. From now on we will write $\overline R:=\widetilde RD^k$.
\end{rem}
\begin{exa}The first differential equation with reflection of which a Green's function was obtained was $x'(t)+mx(-t)=0$ for some $m\in{\mathbb R}$ \cite{Cab4}. This operator is a square root of the harmonic oscillator (in pretty much the same way Dirac's equation does with matrices) and presents very interesting properties. If we think of the analogous operator obtained by substituting $\widetilde D$ by forward difference operator $\Delta=D-\operatorname{Id}$ and $\widetilde \varphi$ by $\varphi$ we get 
$L=\Delta+m\varphi^*=D-\operatorname{Id}+m\varphi^*$. We have that $P=m\operatorname{Id}$, $Q=D-\operatorname{Id}$ and $\overline L=\operatorname{gcd}(m\operatorname{Id},D^{-1}-\operatorname{Id})=\operatorname{Id}$. Therefore, $\widetilde P=P$, $\widetilde Q=Q$ and $\widetilde R=R=\operatorname{Id}-D^{-1}+m\varphi^*$. Thus,
\[ L R= RL=(D-\operatorname{Id}+m\varphi^*)(\operatorname{Id}-D^{-1}+m\varphi^*)=D+D^{-1}+(m^2-2)\operatorname{Id}.\] 
Hence, if $Lu=0$ holds, so does $DRLu=0$ and we get the equation
\[ (D^2+(m^2-2)D+\operatorname{Id})u=0.\] 
The solutions of this equation, for $|m|>2$, are of the form
\[ u_n=c_1 2^{-n} \left(-m^2+|m|\sqrt{m^2-4}+2\right)^n+c_2 2^{-n} \left(-m^2-|m|\sqrt{m^2-4}+2\right)^n\] 
with $c_1,c_2\in{\mathbb R}$. In any case, $Lu=0$ has to hold, so we deduce that \[ c_2=\frac{1}{2} \left(\frac{|m|}{m}\sqrt{m^2-4}+m\right)c_1,\]  and all solutions of $Lu=0$ are expressed as
\[ u_n=c_1 \left[2^{-n} \left(-m^2+|m|\sqrt{m^2-4}+2\right)^n+\frac{1}{2} \left(\frac{|m|}{m}\sqrt{m^2-4}+m\right) 2^{-n} \left(-m^2-|m|\sqrt{m^2-4}+2\right)^n\right],\] 
for some $c_1\in{\mathbb R}$. We can study in an analogous fashion what happens in the case $m\in[-2,2]$.
\end{exa}
\begin{exa}\label{exac} Now instead of substituting $\widetilde D$ by $\Delta$ we do it by $D$, that is, we study the operator $L=D+m\varphi^*$. We have that $P=m\operatorname{Id}$, $Q=D$ and $\overline L=\operatorname{gcd}(m\operatorname{Id},D^{-1})=\operatorname{Id}$. Therefore, $\widetilde P=P$, $\widetilde Q=Q$ and $\widetilde R=R=-D^{-1}+m\varphi^*$. Thus,
\[ RL=L R=(D+m\varphi^*)(-D^{-1}+m\varphi^*)=(m^2-1)\operatorname{Id}.\] 
This means that if the equation $(D+m\varphi^*)u=0$ holds for some $u\in{\mathscr S}$, so does $(m^2-1)u=0$, which is only satisfied if $m=\pm1$. That is, $x_{k+1}-mx_{-k}=0$ is a recurrence relation with reflection with no solution for $m\ne \pm1$. In the case $m=\pm1$, the equation $LR u=0$ is trivial and provides no information on $Lu=0$.

In the case $L=D-\varphi^*$, take $(v_k)_{k\in{\mathbb N}}\subset{\mathbb F}$ arbitrarily and define $u_k=v_k$ if $k\in{\mathbb N}$ and $u_k=u_{1-k}$ if $k\leqslant0$. Clearly $u$ satisfies $Lu=0$. Analogously, if $L=D+\varphi^*$, take $(v_k)_{k\in{\mathbb N}}\subset{\mathbb F}$ arbitrarily and define $u_k=v_k$ if $k\in{\mathbb N}$ and $u_k=-u_{1-k}$ if $k\leqslant0$. $u$ satisfies $Lu=0$. 

\end{exa}

\subsection{Related Operators}

In this section we assume to work over a field of characteristic different from two.

In the theory of differential equations with reflection the even and odd part operators, defined respectively as
\[ (\widetilde Ef)(t):=\frac{f(t)+f(-t)}{2},\ (\widetilde O f)(t):=\frac{f(t)-f(-t)}{2},\] 
play an important role --cf. \cite{Toj3,CTMal,CaTo}. These linear operators satisfy, among others, the properties
\begin{align*}\widetilde E\widetilde D=&\widetilde D\widetilde O,\ \widetilde O\widetilde D=\widetilde D\widetilde E,\ \widetilde E\widetilde\varphi^*=\widetilde\varphi^*\widetilde E=\widetilde E,\ \widetilde O\widetilde\varphi^*= \widetilde\varphi^*\widetilde O=-\widetilde O,\\ \widetilde E+\widetilde O= & \operatorname{Id},\ \widetilde E\widetilde O=\widetilde O\widetilde E=0,\ \widetilde E^2=\widetilde E,\ \widetilde O^2= \widetilde O.\end{align*}

The power of the operators $\widetilde E$ and $\widetilde O$ relies on the fact that they are the projections onto the spaces of even and odd functions respectively.
Now our goal is to take these operators to the setting of ${\mathscr S}$. In order to do this, first observe that the operator $D$ is actually a pullback by the function $\tau(k)=k+1$, $k\in{\mathbb Z}$ and there are precisely two proper invariant subspaces of ${\mathscr S}$ of the map $\tau^2$. They are
\[ {\mathcal E}:=\{u\in{\mathscr S}\ :\ u_{2k+1}=0,\ k\in{\mathbb Z}\},\quad {\mathcal O}:=\{u\in{\mathscr S}\ :\ u_{2k}=0,\ k\in{\mathbb Z}\},\] 
so we actually want to deal with the projections onto those subspaces, which are defined, respectively,
\[ (Eu)_k:=\frac{1+(-1)^k}{2}u_k,\quad (Ou)_k:=\frac{1-(-1)^k}{2}u_k,\] 
for every $(u_k)_{k\in{\mathbb Z}}\in{\mathscr S}$. In order to arrive to $E$ and $O$ we could have used the help of the following map. Let \[ {\mathcal C}_{\mathcal A}:=\{u\in{\mathcal F}({\mathbb Z},{\mathbb C})\ |\ 0\leqslant\limsup_{k\to-\infty} |u_k|^{-\frac{1}{k}}<\limsup_{k\to\infty} |u_k|^{-\frac{1}{k}}\},\]  \[ {\mathcal L}_0:=\{f:B_{\mathbb C}[0,\rho_2]\backslash B_{\mathbb C}(0,\rho_1)\to{\mathbb C}\ |\ \rho_2>\rho_1>0,\ f \text{ is holomorphic}\}.\] 
The elements in ${\mathcal L}_0$ are those holomorphic functions which can be expressed as Laurent series and the elements in ${\mathcal C}_{\mathcal A}$ are the coefficients of those series. Hence, we can consider the bijection
\begin{center}\begin{tikzcd}[row sep=tiny]
{\mathcal C}_{\mathcal A} \arrow{r}{\Xi} & {\mathcal L}_0\\
(u_k)_{k\in{\mathbb Z}} \arrow[mapsto]{r} & \sum\limits_{k\in{\mathbb Z}} u_kx^k
\end{tikzcd}
\end{center}
This way, any operator $\widetilde Y$ on $ {\mathcal L}_0$ (such as can be the even and odd part operators) can be thought as an operator on ${\mathcal C}_{\mathcal A}$ by defining $Y:=\Xi^{-1}\widetilde Y\Xi$. It is easy to check that
\[ E=\Xi^{-1}\widetilde E\Xi,\quad O=\Xi^{-1}\widetilde O\Xi.\] 
\begin{rem}\label{remlam} Observe that $\Lambda=\Xi^{-1}\widetilde \varphi^*\Xi$ is also an involution in ${\mathcal C}_{\mathcal A}$ which is defined as $(\Lambda u)_k=(-1)^ku_k$. In this case $\Lambda$ is not the pullback by any function.
\end{rem}

By definition, it is clear that $E$ and $O$ hold similar properties to $\widetilde E$ and $\widetilde O$:
\[  E D= D O,\ O D= D E, E\varphi^*=\varphi^* E,\ O\varphi^*=\varphi^* O,\ E+ O=\operatorname{Id},\ E O= O E=0,\ E^2= E,\ O^2= O.\] 

We can even combine $E$, $O$, $\widetilde E$ and $\widetilde O$. To do this we can consider $\widetilde E$ and $\widetilde O$ as
\[ \widetilde E=\frac{1}{2}(\operatorname{Id}+\widetilde\varphi^*),\quad \widetilde O=\frac{1}{2}(\operatorname{Id}-\widetilde\varphi^*),\] 
and use the pullback by the inclusion $\iota:{\mathbb Z}\to{\mathbb R}$ to get
\[ \overline E:=\widetilde E\circ\iota^*:=\frac{1}{2}(\operatorname{Id}+\varphi^*),\quad \overline O:=\widetilde O\circ\iota^*:=\frac{1}{2}(\operatorname{Id}-\varphi^*),\] 
defined on ${\mathscr S}$. Thus considered, they have the properties
\begin{align*}\overline E\varphi^*=\varphi^*\overline E=\overline E,\ \overline O\varphi^*= \varphi^*\overline O=-\overline O,\ \overline E+\overline O= \operatorname{Id},\ \overline E\overline O=\overline O\overline E=0,\ \overline E^2=\overline E,\ \overline O^2= \overline O,\end{align*}
but observe that, unlike with $E$ and $O$, the properties $\overline E D= D\overline O$ and $\overline O D= D\overline E$ do not hold.

Observe also that $E$, $O$, $\overline E$ and $\overline O$ commute.

\subsubsection{The exponential map}
In this section we assume to work over a field of characteristic zero.

The reader might have already realized the striking similarity between the algebras ${\mathbb F}[D,D^{-1},\varphi^*]$ and ${\mathbb F}[\widetilde D,\widetilde\varphi^*]$. In fact, as we will see, there is connection between the operators $\widetilde D$ and $\widetilde \varphi^*$ in ${\mathbb F}[D,D^{-1},\varphi^*]$ with, respectively, the operators $D$ and $\varphi$ in ${\mathbb F}[\widetilde D,\widetilde\varphi^*]$ through the exponential map.

To show this, first remember that the exponential of the differential operator is, precisely, the right shift operator, that is, $e^{\widetilde D}=D$ --this fact was shown, symbolically, by Lagrange \cite[p. 13]{Jordan}. 

Observe that the exponential of the derivative at a point $x\in{\mathbb F}$ is, formally,
\[ \delta_x e^{\widetilde D}:=\delta_x \sum_{k=0}^\infty\frac{\widetilde D^k}{k!},\]  
where $\delta_x$ is the Dirac delta distribution at $x$. Consider now the space of analytic functions ${\mathcal A}({\mathbb F})$. Then, for $f\in{\mathcal A}({\mathbb F})$ with a radius of convergence $r>1$ at $x\in{\mathbb F}$,

\[ \delta_x e^{\widetilde D}f=\delta_x \sum_{k=0}^\infty\frac{f^{(k)}}{k!}= \sum_{k=0}^\infty\frac{f^{(k)}(x)}{k!}=\sum_{k=0}^\infty\frac{f^{(k)}(x)}{k!}[(x+1)-x]^k=f(x+1)=\delta_x D f.\] 

So, it is clear that this fact that applies to certain analytic functions can be extended, as a definition, to ${\mathcal F}({\mathbb F},{\mathbb F})$ by defining $e^{\widetilde D}:=D$ and, whenever the exponential of the derivative makes sense as a distribution, it will coincide with our definition. Observe though that this extension is not unique in principle. In order to achieve that we would need to define a topology in ${\mathcal F}({\mathbb F},{\mathbb F})$ such that ${\mathcal A}({\mathbb F})$ is a dense subset.

We could have also shown that $e^{\widetilde D}=D$, formally, using the Fourier transform ${\mathfrak F}$:
\[ {\mathfrak F}^{-1}{\mathfrak F} e^{\widetilde D}={\mathfrak F}^{-1}{\mathfrak F}\left(\sum_{k=0}^\infty\frac{\widetilde D^k}{k!}\right)={\mathfrak F}^{-1}\left( \sum_{k=0}^\infty\frac{(2\pi i x)^k}{k!}\right) {\mathfrak F}={\mathfrak F}^{-1} e^{2\pi ix}{\mathfrak F}=D{\mathfrak F}^{-1}{\mathfrak F}=D,\] 

but this approach cannot be made rigorous due to the fact that $e^{\widetilde D}$ \emph{is not} a distribution. To undertake a proper study of this operator, it has to be done in the framework of \emph{hyperfunctions} \cite[Section 1.3.4]{graf}.

\begin{pro}[{\cite[Proposition 1.6]{graf}}] Let $a\in{\mathbb R}$. We have $e^{a\widetilde D}=D^a$.
\end{pro}

In a similar way, we can compute $e^{a\widetilde \varphi^*}$ for $a\in{\mathbb C}$ taking into account that $\widetilde\varphi|_{\mathbb Z}=\varphi$.
\[ e^{a\widetilde\varphi^*}=\sum_{k=0}^\infty\frac{ (a\widetilde\varphi^*)^k}{k!}= \sum_{k=0}^\infty\frac{a\operatorname{Id}}{(2k)!}+ \sum_{k=0}^\infty\frac{ a\widetilde\varphi^*}{(2k+1)!}=\cosh(a)\operatorname{Id}+\sinh(a)\varphi^*\in {\mathbb C}[D,D^{-1},\varphi^*].\] 

Analogously, we obtain \emph{Euler's formula}:
\[ e^{\widetilde\varphi^*\widetilde D}= \sum_{k=0}^\infty\frac{(\widetilde\varphi^*\widetilde D)^k}{k!}= \sum_{k=0}^\infty\frac{(-1)^k\widetilde D^{2k}}{(2k)!} + \widetilde\varphi^*\sum_{k=0}^\infty\frac{(-1)^k\widetilde D^{2k+1}}{(2k+1)!} =\cos(\widetilde D)+\widetilde\varphi^*\sin(\widetilde D).\] 

Observe that this last expression does not belong to ${\mathbb F}[D,D^{-1},\varphi^*]$. In general, for ${\mathbb F}={\mathbb C}$ and $a\in{\mathbb C}$,
\[ e^{a\widetilde\varphi^*\widetilde D}=\cos(a\widetilde D)+\widetilde\varphi^*\sin(a\widetilde D).\] 

Taking into account that $e^{\widetilde D\widetilde\varphi^*}=e^{-\widetilde\varphi^*\widetilde D}=\cos(\widetilde D)-\widetilde\varphi^*\sin(\widetilde D)$ we have that
\[ \cos(\widetilde D)=\frac{1}{2}\left( e^{\widetilde\varphi^*\widetilde D}+e^{\widetilde D\widetilde\varphi^*}\right) ,\quad \sin(\widetilde D)=\frac{1}{2}\widetilde\varphi^*\left( e^{\widetilde\varphi^*\widetilde D}-e^{\widetilde D\widetilde\varphi^*}\right) .\] 
Analogously, for ${\mathbb F}={\mathbb C}$,
\[ \cosh(\widetilde D)=\frac{1}{2}\left( e^{i\widetilde\varphi^*\widetilde D}+e^{i\widetilde D\widetilde\varphi^*}\right) ,\quad \sinh(\widetilde D)=-\frac{i}{2}\widetilde\varphi^*\left( e^{i\widetilde\varphi^*\widetilde D}-e^{i\widetilde D\widetilde\varphi^*}\right) ,\] 
so
\begin{align*}D= &e^{\widetilde D}= \frac{1}{2}\left( e^{i\widetilde\varphi^*\widetilde D}+e^{i\widetilde D\widetilde\varphi^*}\right) -\frac{i}{2}\widetilde\varphi^*\left( e^{i\widetilde\varphi^*\widetilde D}-e^{i\widetilde D\widetilde\varphi^*}\right) ,\\
D^{-1}= & e^{-\widetilde D}=\frac{1}{2}\left( e^{i\widetilde\varphi^*\widetilde D}+e^{i\widetilde D\widetilde\varphi^*}\right) +\frac{i}{2}\widetilde\varphi^*\left( e^{i\widetilde\varphi^*\widetilde D}-e^{i\widetilde D\widetilde\varphi^*}\right) .\end{align*}
Hence,
\[ D+D^{-1}=e^{i\widetilde\varphi^*\widetilde D}+e^{i\widetilde D\widetilde\varphi^*},\quad D-D^{-1}=-i\varphi^*\left( e^{i\widetilde\varphi^*\widetilde D}-e^{i\widetilde D\widetilde\varphi^*}\right) ,\] 
and therefore $i\varphi^*(D-D^{-1})=e^{i\widetilde\varphi^*\widetilde D}-e^{i\widetilde D\widetilde\varphi^*}$. Thus, we obtain
\[ e^{i\widetilde\varphi^*\widetilde D}=\frac{1}{2}\left( D+D^{-1}+i\varphi^*(D-D^{-1})\right) \in{\mathbb C}[D,D^{-1},\varphi^*].\] 
More generally, for $k\in{\mathbb Z}$,
\[ e^{ik\widetilde\varphi^*\widetilde D}=\frac{1}{2}\left( D^k+D^{-k}+i\varphi^*(D^k-D^{-k})\right) \in{\mathbb C}[D,D^{-1},\varphi^*].\] 

We have shown that, in general, exponentials of the operators in ${\mathbb F}[\widetilde D,\widetilde\varphi^*]$ do not end up in ${\mathcal F}[D,D^{-1},\varphi^*]$, but there are some instances where this is the case and we obtain some interesting relations.
\section{Green's functions}
After the reduction of an operator $L\in{\mathcal F}[D,D^{-1},\varphi^*]$ (Theorem~\ref{T::RL2} and Remark~\ref{remred}) we are left with a recurrence equation of the kind $Sx=0$ with $S\in{\mathbb F}[D]$. In the case of initial conditions it is simple to compute the Green's function. Several results in this direction, stated in different settings, can be found in the classic literature on the subject; see, for instance, \cite[Theorem~11, Chap.~4]{Mill2}, \cite[Section~2.11]{Ag}, \cite[Theorem~2.1]{Mill} or \cite[Theorem~6.8]{Kell}. The differences among these works are due to the operator being studied ($D$, $D^{-1}$ or $\Delta$), whether we consider functions of one real variable (difference equations) or sequences (recurrence relations) as solutions, and the way the authors state the conditions the equation is subject to --see \cite[Section~2.11]{Ag} or \cite[Theorem~2.1]{Mill}. Here we present a version (Theorem~\ref{printh}) which is adequate for our purposes.

\begin{notation}\label{Not::1}
Consider the homogeneous recurrence relation
\begin{equation}
\label{Ec::Shom}(S\,x)_k=\sum_{l=0}^n a_lx_{k+l}=0,\ k\in{\mathbb Z},
\end{equation}	
where $a_l\in{\mathbb F}$, $l=0,\dots,n$ and $a_0a_n\neq 0$. Denote the characteristic polynomial as follows:
\[ p(t)=a_nt^n+a_{n-1}t^{n-1}+\cdots+a_1t+a_0.\] 
Consider the set of different roots of $p$ in $\overline{\mathbb F}$, that is $\{\lambda_1,\dots,\lambda_r\}$ with $r\leqslant n$ and $\lambda_l\neq\lambda_j$ if $l\neq j$.
For each $l\in\{1,\dots,r\}$, denote $h_l$ the multiplicity of the root $\lambda_l$. If $r=n$, then all the roots are of multiplicity $h_l=1$ for $l\in\{1,\dots, n\}$ and the general solution of~\eqref{Ec::Shom} in $\overline{\mathscr S}:={\mathcal F}({\mathbb Z},\overline {\mathbb F})$ is given by
\[ u=k_1y_1+k_2y_2+\cdots+k_ny_n,\] 
where ${(y_l)}_k=\lambda_l^k$, $k_l\in\overline{\mathbb F}$ for $k\in{\mathbb Z}$ and $l=1,\dots,n$.

Now, if $r<n$, then there exists $l\in\{1,\dots,r\}$ such that $h_l>1$. In such a case, the general solution of~\eqref{Ec::Shom} in $\overline{\mathbb F}$ is:
\begin{equation*}u=k_1y_{1,1}+k_2y_{1_2}+\cdots+k_{h_1}y_{1,h_1}+k_{h_1+1}y_{2,1}+\cdots+k_ny_{r,h_r},\end{equation*}
where $\left( y_{l,1}\right) _k=\lambda_l^k$, $\left( y_{l,j}\right) _k=k^{j-1}\lambda_l^k$ for $k\in{\mathbb Z}$, $l\in\{1,\dots,r\}$ and $j\in\{2,\dots,h_l\}$.

If we denote $\Phi=\Big(\begin{array}{ccccccc}
y_{1,1}&y_{1,2}&\cdots&y_{1,h_1}&y_{2,1}&\cdots&y_{r,h_r}
\end{array}\Big)\in{\mathcal M}_{{\mathbb Z}\times n}(\overline{\mathbb F})$ and $K=(k_j)_{j=1}^n\in\overline{\mathbb F} ^n$, we can express the general solution of~\eqref{Ec::Shom} in $\overline{\mathscr S}$ as
$u=\Phi\,K$.
\end{notation}

Observe that, so far, we have obtained solutions in $\overline{\mathscr S}$ not necessarily in ${\mathscr S}$. Nevertheless, we know that, given initial conditions $x_j\in{\mathbb F},\ j=0,\dots, n-1$, by recurrence, problem~\eqref{procfd3} has a unique solution in ${\mathscr S}$. In fact, this means that we can construct $\Phi$ such that $y_{k,j}=\delta_k^j$ for $k,j\in\{0,\dots,n-1\}$ (where $\delta_k^j$ is the Kronecker delta function) just by imposing some the adequate initial conditions.

For the next theorem we define the following disjoint subsets of ${\mathbb Z}^2$ --see Figure~\ref{fig:points}.
\begin{align*} A_1:= &\{(k,j)\in{\mathbb Z}^2\ :\ k> j\geqslant0\}, & A_2:= & \{(k,j)\in{\mathbb Z}^2\ :\ k+1-n\leqslant j<0\},\\
A_3:= &\{(k,j)\in{\mathbb Z}^2\ :\ k+1-n> j,\ j<0\}, & A_4:= &\{(k,j)\in{\mathbb Z}^2\ :\ k\leqslant j,\ j\geqslant0\}.
\end{align*}
Observe that ${\mathbb Z}^2=A_1\sqcup A_2\sqcup A_3\sqcup A_4$.

\begin{figure}[h]
\centering
\includegraphics[width=0.7\linewidth]{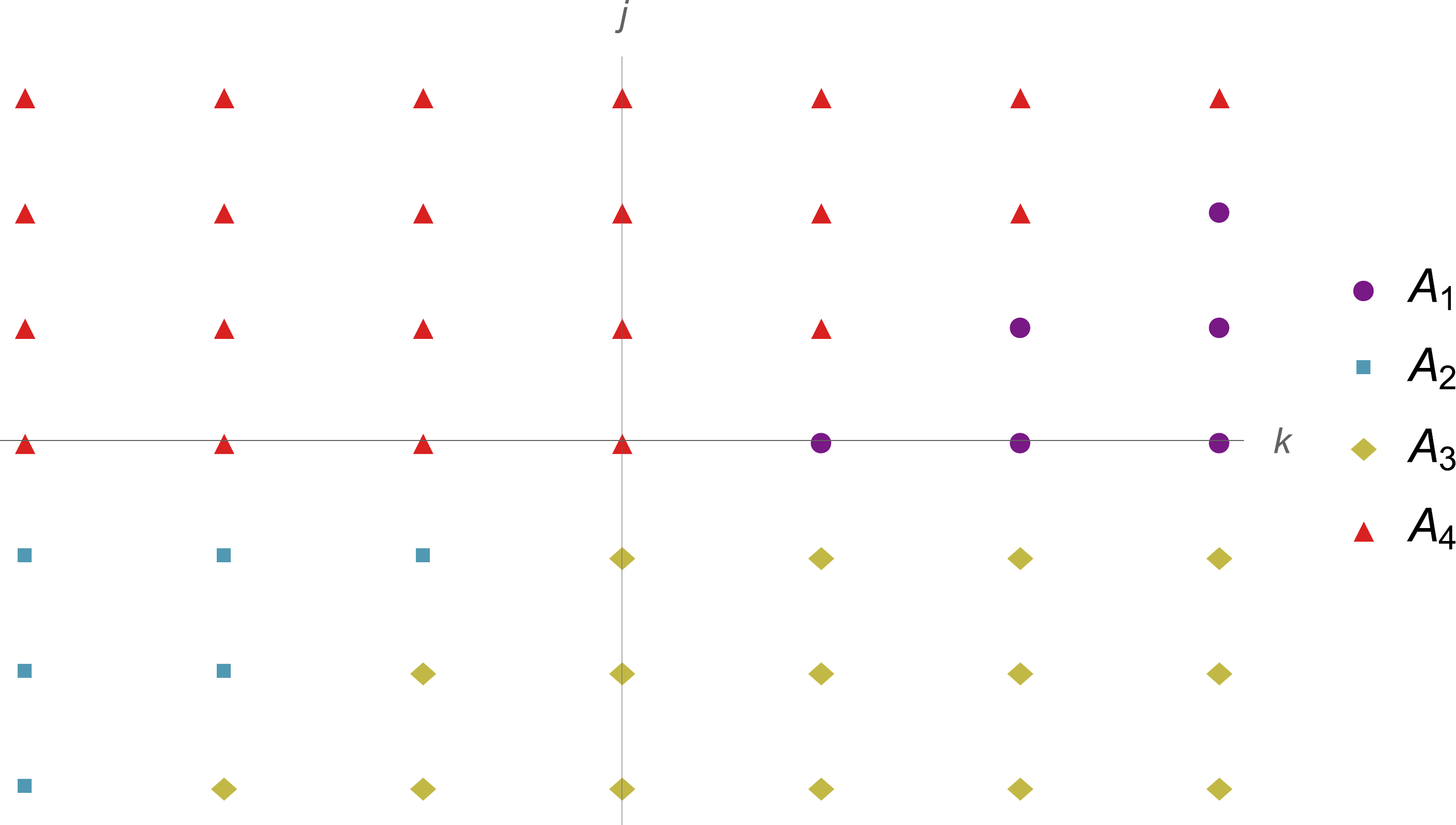}
\caption{The disjoint subsets of ${\mathbb Z}^2$, $A_1,\dots,A_4$.}
\label{fig:points}
\end{figure}

\begin{thm}\label{printh}
Consider the problem
\begin{equation}\label{procfd3}(Sx)_k=\sum_{j=0}^n a_jx_{k+j}=c_k,\ k\in{\mathbb Z},\quad x_j=0,\ j=0,\dots, n-1.
\end{equation}
where $ a_j\in{\mathbb F}$, $j=0,\dots,n$, $ a_0 a_n\ne0$, $c_k\in{\mathbb F}$, $k\in{\mathbb Z}$. Then there is a unique solution of problem~\eqref{procfd3} $u=(u_k)_{k\in{\mathbb Z}}\in{\mathscr S}$ where
\begin{equation*}u_k=\sum_{j\in{\mathbb Z}}\displaystyle H_{k,j}c_{j}\in{\mathbb F},\end{equation*}
 $(H_{k,j})_{k,j\in{\mathbb Z}}\subset{\mathbb F}$ is the \emph{Green's function} given by
\begin{equation}\label{Gfd1}H_{k,j}:=\begin{dcases}
\dfrac{(-1)^{n-1}}{ a_nC_{j+1}}{\widetilde H}_{k,j},& (k,j)\in A_1,\\
\dfrac{1}{ a_0C_{j}}{\widetilde H}_{k,j},& (k,j)\in A_2,\\
0,& (k,j)\in A_3\sqcup A_4,
\end{dcases}\end{equation}
with
\[  {\widetilde H}_{k,j}:=\begin{vmatrix}
y_{1,k} & \cdots & y_{n,k} \\ 
y_{1,j+1} & \cdots & y_{n,j+1} \\ 
y_{1,j+2} & \cdots & y_{n,j+2} \\ 
\vdots & \ddots & \vdots \\ 
y_{1,j+n-1} & \cdots & y_{n,j+n-1} \\ 
\end{vmatrix}, \] 
where $C_j:=\widetilde H_{j,j}$ is the \emph{Casorati} and $\{y_1,\dots,y_n\}$ is a set of fundamental solutions of the homogeneous problem associated to~\eqref{procfd3} such that $y_{k,j}=\delta_k^j$ for $k,j\in\{0,\dots,n-1\}$.
\end{thm}

\begin{proof}
First, by definition of $\{y_1,\dots,y_n\}$, we have that $C_0=1$. Furthermore, we can prove that $C_{j+1}=(-1)^na_nC_j$ for every $j\geqslant0$ \cite[Theorem 3.8]{Mickens}, so $C_{j}\ne 0$ for every $j\in{\mathbb Z}$. By a similar argument, $C_j\ne0$ for every $j<0$. Hence, $H_{k,j}$ is well defined for every $k,j\in{\mathbb Z}$.

 From the definition of $\widetilde H_{k,j}$, we deduce that, for $k\in{\mathbb Z}$,
\begin{equation}\label{Ec::H0} \widetilde H_{k+n,k} =(-1)^{n-1}C_{j+1};\quad \widetilde H_{k+l,k} =0,\ l\in\{0,\dots,n-1\}.
\end{equation}

First, we will see that $\sum_{l=0}^n a_lH_{k+l,j}=\delta_k^j$ for every $k,j\in{\mathbb Z}$. In order to achieve this we will study six different cases --see Figure~\ref{fig:points2}.
\begin{figure}[h]
\centering
\includegraphics[width=0.7\linewidth]{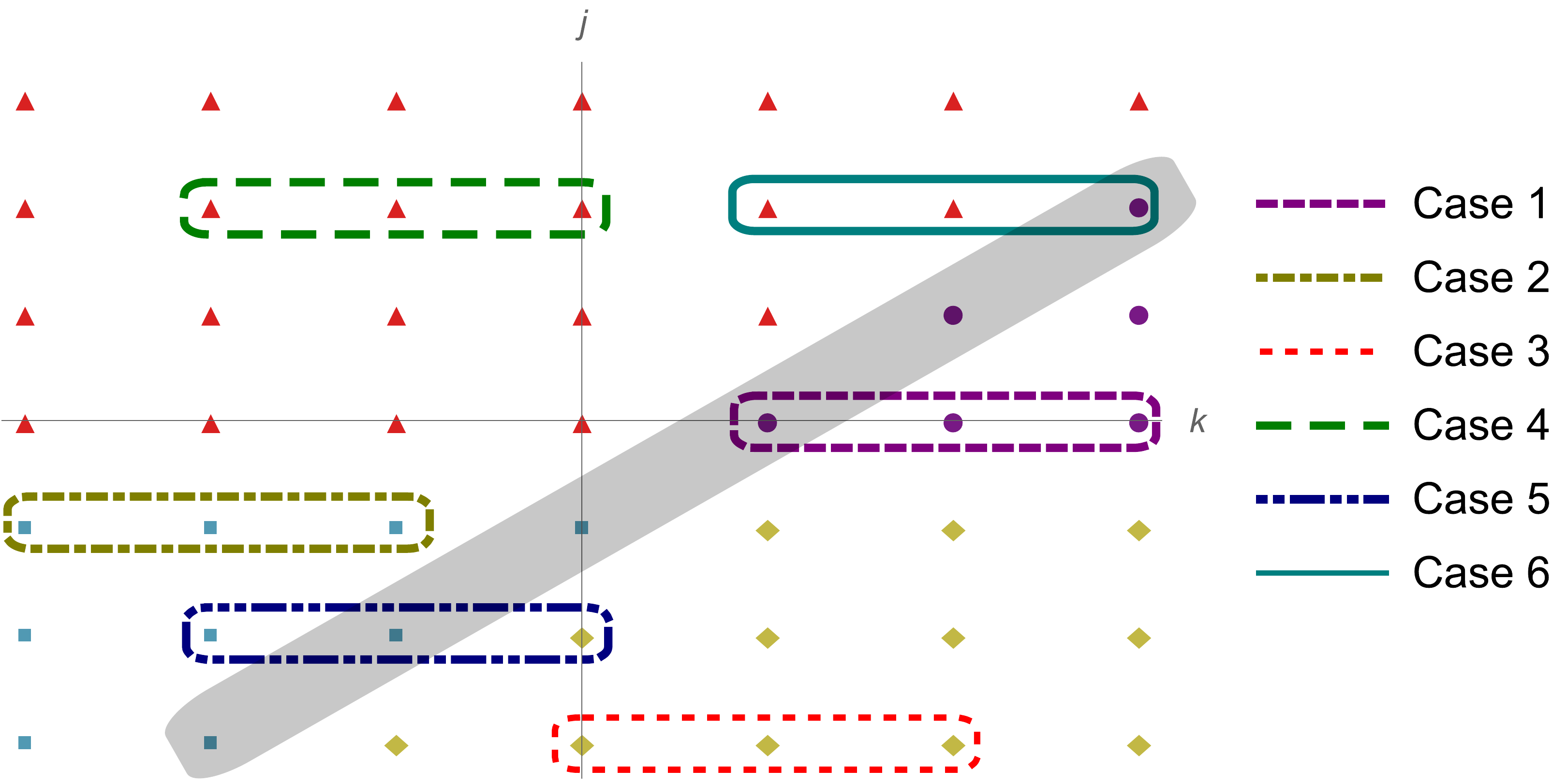}
\caption{Illustration for $n=2$. Each rectangle shows the set of indices $(k+l,j)$ where $l\in\{0,1,2\}$ for $(k,j)$ in one of the six cases. In each case the definition of $H_{k+l,j}$ is different. The shaded area covers those points where $H_{k+l,j}=0$ because of~\eqref{Ec::H0}.}
\label{fig:points2}
\end{figure}

\emph{Case 1:} $(k,j)\in A_1$. In this case $(k+l,j)\in A_1$ for every $l\in\{0,\dots,n\}$ so
\begin{equation*}\sum_{l=0}^n a_lH_{k+l,j}=\frac{(-1)^{n-1}}{ a_nC_{j+1}}\begin{vmatrix}
\sum\limits_{l=0}^n a_ly_{1,k+l} & \cdots & \sum\limits_{l=0}^n a_ly_{n,k+l} \\y_{1,j+1} & \cdots & y_{n,j+1} \\ 
y_{1,j+2} & \cdots & y_{n,j+2} \\ 
\vdots & \ddots & \vdots \\ 
y_{1,j+n-1} & \cdots & y_{n,j+n-1}\\ 
\end{vmatrix}=0.\end{equation*}

\emph{Case 2:} $k+1\leqslant j<0$. In this case $(k,j)\in A_2$ and $(k+l,j)\in A_2$ for every $l\in\{0,\dots,n\}$, so
\begin{equation*}\sum_{l=0}^n a_lH_{k+l,j}=\frac{1}{ a_0C_{j+n-2}}\begin{vmatrix}
\sum\limits_{l=0}^n a_ly_{1,k+n-2+l} & \cdots & \sum\limits_{l=0}^n a_ly_{n,k+n-2+l} \\ 
y_{1,j+n-1} & \cdots & y_{n,j+n-1} \\ 
y_{1,j+n} & \cdots & y_{n,j+n} \\ 
\vdots & \ddots & \vdots \\ 
y_{1,j+2n-3} & \cdots & y_{n,j+2n-3}\\
\end{vmatrix}=0.\end{equation*}

\emph{Case 3:} $(k,j)\in A_3$. In this case $(k+l,j)\in A_3$ for every $l\in\{0,\dots,n\}$ so $H_{k+l,j}=0$ for every $l\in\{0,\dots,n\}$.

\emph{Case 4:} $k+n\leqslant j,\ j\geqslant0$. In this case $(k,j)\in A_4$ and $(k+l,j)\in A_4$ for every $l\in\{0,\dots,n\}$ so $H_{k+l,j}=0$ for every $l\in\{0,\dots,n\}$.

\emph{Case 5:} $k\in\{j,\dots,j+n-1\}$ and $j<0$. in this case $(k+l,j)\in A_2$ for $l\in\{0,\dots,j-k-1+n\}$ and $(k+l,j)\in A_4$ for $l\in\{j-k+n,\dots,n\}$. Hence, using~\eqref{Ec::H0},
\begin{align*}
\sum_{l=0}^n a_lH_{k+l,j}= & \sum_{l=0}^{j-k+n-1} a_lH_{k+l,j}=\sum_{l=0}^{j-k+n-1}\dfrac{ a_l}{ a_0C_{j}}{\widetilde H}_{k+l,j}=\sum_{l=0}^{j-k+n-1}\dfrac{ a_l}{ a_0C_{j}}{\widetilde H}_{j+(k-j+l),j} \\ = & \sum_{m=k-j}^{n-1}\dfrac{ a_{m+j-k}}{ a_0C_{j}}{\widetilde H}_{j+m,j}=\sum_{m=k-j}^{0}\dfrac{ a_{m+j-k}}{ a_0C_{j}}{\widetilde H}_{j+m,j}.
\end{align*}
Since $k-j\geqslant 0$, this last expression is $0$ if $k>j$ and, otherwise, $k=j$ and
\[ \sum_{l=0}^n a_lH_{k+l,j}=\dfrac{ a_{0}}{ a_0C_{j}}{\widetilde H}_{j,j}=1.\] 

\emph{Case 6:} $k\in\{j-n,\dots,j\}$ and $j\geqslant0$. in this case $(k+l,j)\in A_4$ for $l\in\{0,\dots,j-k\}$ and $(k+l,j)\in A_1$ for $l\in\{j-k+1,\dots,n\}$. Since $k\in\{j-n,\dots,j\}$, we have that $n-j+k+1\in\{0,\dots,n\}$ and, therefore, using~\eqref{Ec::H0},
\begin{align*}
\sum_{l=0}^n a_lH_{k+l,j}= &\sum_{l=j-k+1}^{n} a_lH_{k+l,j}=\sum_{l=j-k+1}^{n} a_l\dfrac{(-1)^{n-1}}{ a_nC_{j+1}}{\widetilde H}_{k+l,j}=\sum_{m=1}^{n-j+k} a_{j-k+m}\dfrac{(-1)^{n-1}}{ a_nC_{j+1}}\widetilde H_{j+m,j} \\ = & \sum_{m=n}^{n-j+k} a_{j-k+m}\dfrac{(-1)^{n-1}}{ a_nC_{j+1}}\widetilde H_{j+m,j}.
\end{align*}
This last expression is $0$ if $k<j$ and, otherwise, $k=j$ and 
\[ 
\sum_{l=0}^n a_lH_{k+l,j}= a_{n}\dfrac{(-1)^{n-1}}{ a_nC_{j+1}}\widetilde H_{j+n,j}=\dfrac{(-1)^{n-1}}{C_{j+1}}\widetilde H_{j+n,j}=1.\] 
Hence,
\[ \sum_{l=0}^n a_lH_{k+l,j}=\delta_k^j.\] 

Now we have that
\begin{align*}(Su)_k = & \sum_{l=0}^n a_l\left(\sum_{j\in{\mathbb Z}} H_{k+l,j}c_{j}\right)=\sum_{j\in{\mathbb Z}}\left( \sum_{l=0}^n a_l H_{k+l,j}\right) c_{j}=\sum_{j\in{\mathbb Z}}\delta_k^jc_{j}=c_k.
\end{align*}
Furthermore, for $k=0,\dots,n-1$ and $j<0$, either $(k,j)\in A_3$, and hence $H_{k,j}=0$, or $(k,j)\in A_2$, in which case $0\leqslant k\leqslant j+n-1$. Hence, $0<-j\leqslant k-j\leqslant n-1$, so $H_{k,j}=H_{j+(k-j),j}=0$. Thus, we can write
\[ u_k=\sum_{j\in{\mathbb Z}} H_{k,j}c_{j}=\sum_{j\geqslant 0} H_{k,j}c_{j}=\sum_{j= 0}^{k-1} H_{k,j}c_{j}=\sum_{j= 0}^{k-1} H_{j+(k-j),j}c_{j}=0.\] 
This last equality holds because $k-j\in\{1,\dots,k\}\subset\{1,\dots,n-1\}$ for any $j\in\{0,\dots,k-1\}$. Therefore, $u$ is a solution of problem~\eqref{procfd3}.

Finally, it is left to show that $(H_{k,j})_{k,j\in{\mathbb Z}}\subset{\mathbb F}$, but this has to be so because we already new, by recurrence, that problem~\eqref{procfd3} had a unique solution in ${\mathscr S}$. Hence, fix $j\in{\mathbb Z}$ and take $c_k=\delta_k^j$ for every $k\in{\mathbb Z}$. We have that $u_k=H_{k,j}\in{\mathbb F}$ for every $k\in{\mathbb Z}$, which ends the proof.
\end{proof}
\begin{rem} Similar results appear in the context of non-homogeneous generalized linear discrete time systems (see \cite[Corollary 3.1]{Da1} for a result in the field of order $n$ systems obtained through matrix pencil theory), or linear non-autonomous fractional $\nabla$-difference equations \cite[Theorem 2.1]{Da2}.
\end{rem}
Let us consider $H\in{\mathcal M}_{{\mathbb Z}\times {\mathbb Z}}$ defined as follows:
\begin{equation}
\label{Ec::DH} (H)_{k,j}=H_{k,j},\quad k,\,j\in{\mathbb Z},
\end{equation}
where $H_{k,j}$ is defined in~\eqref{Gfd1} for each $k$, $j\in{\mathbb Z}$. Using this notation, we can rewrite the previous result in a vectorial way.

\begin{cor}
	Consider the problem
	\begin{equation}\label{procfd3vec}Sx=\sum_{j=0}^na_jD^jx=c,\ x\in{\mathscr S},\quad (x)_j=0,\ j=0,\dots, n-1.
	\end{equation}
	where $a_j\in{\mathbb F}$, $j=0,\dots,n$, $a_0a_n\ne0$, $b\in {\mathscr S}$. Then there is a unique solution of problem~\eqref{procfd3vec} given by $u=H\,c$, 
	where $H$ is the \emph{Green's function} defined in~\eqref{Ec::DH}.
\end{cor}

\subsection{General boundary conditions}

From now on, given a vector space $V$ we denote by $V^*$ its algebraic dual. Consider the vector space ${\mathcal T}_n$ generated by those solutions of order $n$ equations of the form~\eqref{Ec::Shom}, that is

 \[ {\mathcal T}_n=\left\{\left( \sum_{j=1}^p\alpha_jk^{n_j}z_j^{k}\right) _{k\in{\mathbb Z}}\in{\mathscr S}\ : z_j\in\overline{\mathbb F},\ n_j\in\{0,1,\dots,n\},\ \alpha_j\in{\mathbb F};\ j=1,\dots,p;\ p\in{\mathbb N} \right\}.\] 
Observe that, by asking the sum to be in ${\mathscr S}$, we are assuming values in ${\mathbb F}$. Also, For every $L\in{\mathbb F}[D,D^{-1},\varphi^*]$, we have that $L(f)\in{\mathcal T}_n$ for every $f\in{\mathcal T}_n$.
\begin{cor}\label{corref}
Let $W\in ({\mathcal T}_n^*)^n$ and consider the problem
\begin{equation}\label{procfd4}\sum_{j=0}^na_jx_{k+j}=c_k,\ k\in{\mathbb Z},\quad Wx=h.
\end{equation}
where $a_j\in{\mathbb F}$, $j=0,\dots,n$, $a_0a_n\ne0$, $c_k\in{\mathbb F}$, $k\in{\mathbb Z}$, $h\in{\mathbb F}^n$. Then there is a unique solution of problem~\eqref{procfd4} in ${\mathcal T}_n$ if, and only if, $\det(W_\Phi)\neq 0$, where $W_\Phi:=W\,\Phi\in{\mathcal M}_{n}({\mathbb F})$, with $\Phi$ defined in Notation~\ref{Not::1}.

In such a case, the unique solution is given by:
\[ u=\Phi\,W_\Phi^{-1}h+\left( H-\Phi\,W_\Phi^{-1}W\,H\right) c,\] 
where $H$ is defined in~\eqref{Ec::DH} assuming $WHc$ is well defined.
\end{cor}

\begin{proof}
	Every solution of $S\,x=c$ is given by
	\[ u=\Phi\,K+H\,c,\quad K\in{\mathbb R}^n.\] 

	If we impose the condition given by $W$, we have the order $n$ linear system of equations
	\[ W\,u=W\,\Phi\,K+W\,H\,c=h.\] 

It is clear that there exist a unique solution of previous system if, and only if, \[ \det(W\,\Phi)=\det(W_\Phi)\neq 0,\] 

In such a case:
\[ K=W_\Phi^{-1}h-W_\Phi^{-1}W\,H\,c,\] 
thus
\[ u=\Phi\,W_\Phi^{-1}h+\left( H-\Phi\,W_\Phi^{-1}W\,H\right) c,\] 
and the result is proved.
\end{proof}
 \begin{rem}
	In Corollary~\ref{corref} we had to ask the compatibility between the boundary conditions and the equation in two instances. Fist, by asking for $W$ to be in $({\mathcal T}_n^*)^n$, $W\Phi$ was well defined. Second, the compatibility with the nonlinear part of the equations was guaranteed by asking that $WHc$ were well defined.
 	In the first case it would be enough for $W$ to be in the dual of the vector space of the solutions of the homogeneous problem associated to~\eqref{procfd4}, but this would require to know them aforehand.
 	\end{rem}

\begin{cor}Let $W\in ({\mathcal T}_n^*)^n$. Consider the problem 
\begin{equation}\label{rbvp2}Lx=c,\ Wx=h.
\end{equation}
where $L$ is defined as in~\eqref{difeqin}. Then, there exists $\overline R\in {\mathbb F}[D,\varphi^* ]$ --as in Remark~\ref{remred}-- such that $L\overline R\in{\mathbb F}[D]$ and a solution of problem~\eqref{rbvp2} is given by \[ u:= \Phi\,W_\Phi^{-1}h+\left( \overline RH-\Phi W_\Phi^{-1}W\overline RH\right) c\]  where $H$ is a Green's function associated to the problem
\begin{equation}\label{prosim}L\overline Rx=c,\ Wx=W\overline Rx=0,\end{equation} assuming it exists, $W\overline RHc$ is well defined, $\Phi$ is the general solution of $L\overline Rx=0$ and $W_\Phi:=W\,\Phi\in{\mathcal M}_{n}({\mathbb F})$ is invertible.
\end{cor}
\begin{proof} First, we have that, since $L\Phi=0$,
\[ Lu =L\left( \Phi\,W_\Phi^{-1}h+\left( \overline RH-\Phi W_\Phi^{-1}W\overline RH\right) c\right) =\left( L\overline R\right) (H c)=\operatorname{Id} c=c.\] 
On the other hand, since $H$ is the Green's function of problem~\eqref{prosim}, it has to satisfy the boundary conditions, that is $WH=W\overline R H=0$ (to see this just take $h=(\delta_j^k)_{j\in{\mathbb Z}}$ for $k\in{\mathbb Z}$). Hence,
\[ Wu=W\left( \Phi\,W_\Phi^{-1}h+\left( \overline RH-\Phi W_\Phi^{-1}W\overline RH\right) c\right) =h,\] 
 so $u$ is a solution of problem~\eqref{rbvp2}.

\end{proof}

In the next section we will talk about systems, which will allow us to illustrate those cases where we can guarantee the solution of problem~\eqref{rbvp2} is unique.

\section{Linear systems of difference equations with reflection}

In this section we will consider the homogeneous system of linear difference equations
\begin{equation}\label{hlsystem}(Ju)_k:=Fx_{k+1}+Gx_{-k-1}+A x_k+Bx_{-k}=0,\ k\in{\mathbb Z},
\end{equation}
where $x_k\in{\mathbb F}^n$, $n\in{\mathbb N}$, $A,B,F,G\in{\mathcal M}_n({\mathbb F})$ and $u\in{\mathcal F}({\mathbb Z},{\mathbb F}^n)$. We will prove that a fundamental matrix for problem~\eqref{hlsystem} exists.

\begin{dfn} We say that
$M\in{\mathcal F}({\mathbb Z},M_n({\mathbb F}))$ is a \textit{fundamental matrix} of problem~\eqref{hlsystem} if $(u_k)_{k\in{\mathbb Z}}=(M(k)\,u_0))_{k\in{\mathbb Z}}$ is a solution of equation~\eqref{hlsystem} for every $u_0\in{\mathbb F}^n$, that is
\begin{equation*}FM(k+1)+GM(-k-1)+A M(k)+BM(-k)=0,\ k\in{\mathbb Z}.
\end{equation*}
\end{dfn}

\begin{dfn} If $M$ is a block matrix of the form
\[ M=\left( \begin{array}{c|c}
M_1 & M_2\\ \hline
M_3 & M_4
\end{array}\right) ,\] 
where $M_k\in{\mathcal M}_n({\mathbb F})$, we define $M_{(k)}:=M_k$.
\end{dfn}

\begin{thm}\label{thmexpfm}
Assume that
\[ \left( \begin{array}{c|c}
F & G\\ \hline
B & A
\end{array}\right) \text{ and } \left( \begin{array}{c|c}
A & B\\ \hline
G & F
\end{array}\right) \] 
are invertible. Then
\begin{equation*}M:=\left( \left[-\left( \begin{array}{c|c}
F & G\\ \hline
B & A
\end{array}\right) ^{-1}\left( \begin{array}{c|c}
A & B\\ \hline
G & F
\end{array}\right) \right]^k_{(1)}+\left[-\left( \begin{array}{c|c}
F & G\\ \hline
B & A
\end{array}\right) ^{-1}\left( \begin{array}{c|c}
A & B\\ \hline
G & F
\end{array}\right) \right]^k_{(2)}\right) _{k\in{\mathbb Z}}.\end{equation*}
 is a fundamental matrix of problem~\eqref{hlsystem}. Furthermore, problem~\eqref{hlsystem} equipped with the boundary condition $x_0=u_0\in{\mathbb F}^n$ has a unique solution given by $(u_k)_{k\in{\mathbb Z}}=(M(k)\,u_0))_{k\in{\mathbb Z}}$.
\end{thm}

\begin{proof}
If we define $v=\varphi^*u$, then we have that problem~\eqref{hlsystem} can be expressed as
\[ FDu+GDv+Au+Bv=0.\] 
Composing with $\varphi^*$, we get
\[ FD^{-1}\varphi^*u+GD^{-1}\varphi^*v+A\varphi^*u+B\varphi^*v=FD^{-1}v+GD^{-1}u+Av+Bu=0.\] 
Now, composing with $D$,
\[ Fv+Gu+ADv+BDu=0.\] 
Hence, we have the system
\[ \left( \begin{array}{c|c}
F & G\\ \hline
B & A
\end{array}\right) \left( \begin{array}{c}
Du \\ \hline
Dv 
\end{array}\right) =-\left( \begin{array}{c|c}
A & B\\ \hline
G & F
\end{array}\right) \left( \begin{array}{c}
u \\ \hline
v 
\end{array}\right) .\] 
The hypotheses of the theorem regarding the invertibility of the matrices imply that this is a regular system,
so we can solve for $Du$ and $Dv$ in the following way:
\begin{equation*}\left( \begin{array}{c}
Du \\ \hline
Dv 
\end{array}\right) =-\left( \begin{array}{c|c}
F & G\\ \hline
B & A
\end{array}\right) ^{-1}\left( \begin{array}{c|c}
A & B\\ \hline
G & F
\end{array}\right) \left( \begin{array}{c}
u \\ \hline
v 
\end{array}\right) .\end{equation*}
In particular, iterating,
\[ \left( \begin{array}{c}
u_k \\ \hline
v_k 
\end{array}\right) =\left[-\left( \begin{array}{c|c}
F & G\\ \hline
B & A
\end{array}\right) ^{-1}\left( \begin{array}{c|c}
A & B\\ \hline
G & F
\end{array}\right) \right]^k\left( \begin{array}{c}
u_0 \\ \hline
v _0
\end{array}\right) =\left[-\left( \begin{array}{c|c}
F & G\\ \hline
B & A
\end{array}\right) ^{-1}\left( \begin{array}{c|c}
A & B\\ \hline
G & F
\end{array}\right) \right]^k\left( \begin{array}{c}
u_0 \\ \hline
u _0
\end{array}\right) ,\] 
for $k\geqslant1$. Therefore,
\[ u_k=\left( \left[-\left( \begin{array}{c|c}
F & G\\ \hline
B & A
\end{array}\right) ^{-1}\left( \begin{array}{c|c}
A & B\\ \hline
G & F
\end{array}\right) \right]^k_{(1)}+\left[-\left( \begin{array}{c|c}
F & G\\ \hline
B & A
\end{array}\right) ^{-1}\left( \begin{array}{c|c}
A & B\\ \hline
G & F
\end{array}\right) \right]^k_{(2)}\right) u_0,\] 
for $k\geqslant1$. We can proceed analogously for $k\leqslant-1$, since
\[ \left( \begin{array}{c}
D^{-1}u \\ \hline
D^{-1}v 
\end{array}\right) =-\left( \begin{array}{c|c}
A & B\\ \hline
G & F
\end{array}\right) ^{-1}\left( \begin{array}{c|c}
F & G\\ \hline
B & A
\end{array}\right) \left( \begin{array}{c}
u \\ \hline
v 
\end{array}\right) .\] 
Hence, we have the result.
\end{proof}

The next theorem serves to construct the Green's function of a system of recurrence relations on ${\mathbb Z}$. The reader may consult \cite{Mill} for more information on the subject in the context of systems of recurrence relations on ${\mathbb N}$ with nonconstant coefficients.
\begin{thm} Consider a system of recurrence relations of the form
 \begin{equation}\label{ns}x_{k+1}=Kx_k,\ k\in{\mathbb Z},\end{equation}
 where $x_k\in{\mathbb F}^n$ and $K\in{\mathcal M}_n({\mathbb F})$ is invertible. Define
\[ \overline H_{k,j}:=\begin{dcases} K^{k-1-j}, & -1\leqslant j\leqslant k-1,\\
- K^{k-1-j}, & k\leqslant j\leqslant -1,\\ 0, & \text{otherwise.}\end{dcases}
\] 
Then $\overline H:=(\overline H_{k,j})_{k,j\in{\mathbb Z}}$ is a Green's function of problem~\eqref{ns}, that is, a solution of
\begin{equation*}x_{k+1}=Kx_k+c_k,\ k\in{\mathbb Z}.
\end{equation*}
where $c=(c_k)_{k\in{\mathbb Z}}\in{\mathcal F}({\mathbb Z},{\mathbb F}^n)$ is given by $u=\overline Hc$.
\end{thm}
\begin{proof} Let $u:=\overline Hc$. Then, for $k\geqslant0$,
\begin{align*} (Du)_k-(Ku)_k= &(D\overline Hc)_k-K(\overline Hc)_k=D\left( \sum_{j=-1}^{k-1}K^{k-1-j}c_j\right) _k-K\sum_{j=-1}^{k-1}K^{k-1-j}c_j\\ =& \sum_{j=-1}^{k}K^{k-j}c_j-\sum_{j=-1}^{k-1}K^{k-j}c_j=c_k.
\end{align*}
Analogously, for $k\leqslant-1$,
\begin{align*} (Du)_k-(Ku)_k= &(D\overline Hc)_k-K(\overline Hc)_k=-D\left( \sum_{j=k}^{-1}K^{k-1-j}c_j\right) _k+K\sum_{j=k}^{-1}K^{k-1-j}c_j\\ =& -\sum_{j=k+1}^{-1}K^{k-j}c_j+\sum_{j=k}^{-1}K^{k-j}c_j=c_k.
\end{align*}

\end{proof}
\begin{thm}\label{thmfin}Consider $J$ as defined in~\eqref{hlsystem} and 
assume that
\[ \left( \begin{array}{c|c}
F & G\\ \hline
B & A
\end{array}\right) \text{ and } \left( \begin{array}{c|c}
A & B\\ \hline
G & F
\end{array}\right) \] 
are invertible. Consider the problem
\begin{equation}\label{procfd5}Jx=c,\quad Wx=h.
\end{equation}
Then the sequence given by
\[ u=\pi_1\left( XZ^{-1}\left[\left( \begin{array}{c}
h \\ \hline
h 
\end{array}\right) -\left( \begin{array}{c}
W \\ \hline
W\varphi^* 
\end{array}\right) Y\right]+Y\right) ,\] 
where 
\[ X:=\left( \left[-\left( \begin{array}{c|c}
F & G\\ \hline
B & A
\end{array}\right) ^{-1}\left( \begin{array}{c|c}
A & B\\ \hline
G & F
\end{array}\right) \right]^k\right) _{k\in{\mathbb Z}},\ Y:=\overline H\left( \begin{array}{c|c}
F & G\\ \hline
B & A
\end{array}\right) ^{-1}\left( \begin{array}{c}
c \\ \hline
\varphi^*c 
\end{array}\right) ,\ Z:=\left( \begin{array}{c}
W \\ \hline
W\varphi^* 
\end{array}\right) X,\] 
 $\overline H$ is the Green's function of problem~\eqref{syst2} and $\pi_1:{\mathbb F}^n\times{\mathbb F}^n\to{\mathbb F}^n$ is such that $\pi_1(x,y)=x$, is the unique solution of problem~\eqref{procfd5}, provided all of the terms involved are well defined and $Z$ is invertible.
\end{thm}

\begin{proof}
Proceeding as in the proof of Theorem~\ref{thmexpfm}, we can reduce the equation $Jx=c$ to 
\begin{equation}\label{syst2}\left( \begin{array}{c}
Du \\ \hline
Dv 
\end{array}\right) =-\left( \begin{array}{c|c}
F & G\\ \hline
B & A
\end{array}\right) ^{-1}\left( \begin{array}{c|c}
A & B\\ \hline
G & F
\end{array}\right) \left( \begin{array}{c}
u \\ \hline
v 
\end{array}\right) +\left( \begin{array}{c|c}
F & G\\ \hline
B & A
\end{array}\right) ^{-1}\left( \begin{array}{c}
c \\ \hline
D\varphi^*c 
\end{array}\right) .\end{equation}
A particular solution of~\eqref{syst2} can be expressed as
\[ \left( \begin{array}{c}
u \\ \hline
v 
\end{array}\right) =\overline H\left( \begin{array}{c|c}
F & G\\ \hline
B & A
\end{array}\right) ^{-1}\left( \begin{array}{c}
c \\ \hline
\varphi^*c 
\end{array}\right) 
\] 
so the general solution of~\eqref{syst2} is of the form
\[ \left( \begin{array}{c}
u \\ \hline
v 
\end{array}\right) =\left( \left[-\left( \begin{array}{c|c}
F & G\\ \hline
B & A
\end{array}\right) ^{-1}\left( \begin{array}{c|c}
A & B\\ \hline
G & F
\end{array}\right) \right]^k\right) _{k\in{\mathbb Z}}r+\overline H\left( \begin{array}{c|c}
F & G\\ \hline
B & A
\end{array}\right) ^{-1}\left( \begin{array}{c}
c \\ \hline
D\varphi^*c 
\end{array}\right) 
\] 
with $r\in{\mathbb F}^{2n}$. Then, imposing $Wu=h$, and thus $W\varphi^*v=h$,
\begin{align*}\left( \begin{array}{c}
h \\ \hline
h 
\end{array}\right) = & \left( \begin{array}{c}
W u \\ \hline
W\varphi^*v 
\end{array}\right) =\left( \begin{array}{c}
W \\ \hline
W\varphi^* 
\end{array}\right) Xr +\left( \begin{array}{c}
W \\ \hline
W\varphi^* 
\end{array}\right) Y
\end{align*}
Hence, this system can only be solved uniquely if $Z$ is a regular matrix. Therefore,
\[ r=Z^{-1}\left[\left( \begin{array}{c}
h \\ \hline
h 
\end{array}\right) -\left( \begin{array}{c}
W \\ \hline
W\varphi^* 
\end{array}\right) Y\right].\] 
That is,

\[ \left( \begin{array}{c}
u \\ \hline
v 
\end{array}\right) =XZ^{-1}\left[\left( \begin{array}{c}
h \\ \hline
h 
\end{array}\right) -\left( \begin{array}{c}
W \\ \hline
W\varphi^* 
\end{array}\right) Y\right]+Y.
\] 
Thus,
\[ u=\pi_1\left( XZ^{-1}\left[\left( \begin{array}{c}
h \\ \hline
h 
\end{array}\right) -\left( \begin{array}{c}
W \\ \hline
W\varphi^* 
\end{array}\right) Y\right]+Y\right) .\] 

\end{proof}

\begin{cor}\label{corfin} Assume $a_na_{-n}-b_nb_{-n}\ne 0$. If the problem
\begin{equation}\label{rbvp3}\sum_{j=-n}^n\left( a_jx_{k+j}+b_jx_{-k-j}\right) =c_k,\ k\in{\mathbb Z};\quad x_k=\xi_k,\ k=1,\dots,n,
\end{equation}
 has a solution, it is unique.
\end{cor}
\begin{proof}
Define $y_k=(x_{k-n},\dots,x_{k+n-1})$. Denote by $y_{k,j}$ the $j$-th component of $y_k$ (starting at $j=-n$) and by $y_{\cdot,j}$ the sequence $(y_{k,j})_{k\in{\mathbb Z}}$. Then, we have that $Dy_{\cdot,j}=y_{\cdot,j+1}$ for $j=-n,\dots,n-1$ and
\[ c_k=\sum_{j=-n}^n\left( a_jy_{k,j}+b_j\varphi^*y_{k,-j}\right) ,\ k\in{\mathbb Z}.\] 
Now, define $c=(c_k)_{k\in{\mathbb Z}}$ and $A,B,F,G\in{\mathcal M}_{2n}({\mathbb R})$ such that
\[ F=\left( \begin{array}{c|c}
\operatorname{Id} &\bm 0 \\ \hline
\bm 0 & a_{n}
\end{array}\right) ,\ G=\left( \begin{array}{c|c}
\bm 0 & \bm 0 \\ \hline
\bm 0 & b_{n}
\end{array}\right) 
 ,\] \[ A=\begin{pmatrix}
0 & -1 & 0 & \cdots & 0 & 0 \\ 
0 & 0 & -1 & \cdots & 0 & 0 \\ 
\vdots & \vdots & \vdots & \ddots & \vdots & \vdots \\ 
0 & 0 & 0 & \cdots & -1 & 0 \\ 
0 & 0 & 0 & \cdots & 0 & -1 \\ 
a_{-n} & a_{-n+1} & a_{-n+2} & \cdots & a_{n-2} & a_{n-1}
\end{pmatrix},\ B=\begin{pmatrix}
0 & \cdots & 0 \\ 
\vdots & \ddots & \vdots \\ 
0 & \cdots & 0 \\ 
b_{-n} & \cdots & b_{n-1}
\end{pmatrix}, 
\] 
where $\bm 0$ denotes a zero matrix. We have that problem~\eqref{rbvp3} can be expressed in the form of system~\eqref{hlsystem}, that is, \begin{equation}\label{hlsystem2}Fy_{k+1}+Gy_{-k-1}+A y_k+By_{-k}=c_k,\ k\in{\mathbb Z},\quad y_0=\xi,
\end{equation}
where $\xi=\left( \xi_1,\dots,\xi_n\right) $. Now, by \cite[Lemma 3.8]{CaTo}, we have that
\begin{align*}& \left|\begin{array}{c|c}
F & G\\ \hline
B & A
\end{array}\right|= \left|\begin{array}{c|c}
A & B\\ \hline
G & F
\end{array}\right|=|FA-BG|\\= &\begin{vmatrix}
0 & -1 & 0 & \cdots & 0 & 0 \\ 
0 & 0 & -1 & \cdots & 0 & 0 \\ 
\vdots & \vdots & \vdots & \ddots & \vdots & \vdots \\ 
0 & 0 & 0 & \cdots & -1 & 0 \\ 
0 & 0 & 0 & \cdots & 0 & -1 \\ 
a_na_{-n}-b_nb_{-n} & a_na_{-n+1}-b_nb_{-n+1} & a_na_{-n+2}-b_nb_{-n+2} & \cdots & a_na_{n-2}-b_nb_{n-2} & a_na_{n-1}-b_nb_{n-2}
\end{vmatrix}\\ = & a_na_{-n}-b_nb_{-n}\ne 0.
\end{align*}

On the other hand, $W$ acts on $y$ as evaluating $y$ on $0$, so
\[ Z:=\left( \begin{array}{c}
W \\ \hline
W\varphi^* 
\end{array}\right) X=\operatorname{Id},\] 
is invertible. Hence, applying Theorem~\ref{thmfin}, we conclude that the system~\eqref{hlsystem2} has a unique solution and thus so does problem~\eqref{rbvp3}.
\end{proof}
\begin{rem} Observe that the problem in Example~\ref{exac} fails to meet the hypotheses of Corollary~\ref{corfin}.
\end{rem}

\section{Conclusions and open problems}

Throughout this work we have developed a theory of linear recurrence equations and systems with reflection and constant coefficients. Most of the theory is valid for fields of arbitrary characteristic. We would have to avoid division dividing by $0$, for instance when defining the operators $\widetilde E$, $\widetilde O$, $E$ and $O$. For more information on recurrence relations on fields in arbitrary characteristic the reader may consult \cite{Ivanov,Conrad}.

There are some clear ways in which the theory could be extended. We point out here some of them.
\begin{itemize}
\item \emph{Non-constant coefficients:} The theory of linear differential equations with constant coefficients (and its difference counterpart) is basically the same than in the constant coefficient case. The main difference in the case of systems is whether a fundamental matrix can be obtained explicitly by taking the exponential of the matrix $A(t)$ defining the system, something which is true if $A(t)A(s)=A(s)A(t)$ for every $t$ and $s$ \cite{CabToj,Kotin}. Unfortunately, this happens under very restrictive circumstances \cite{CabToj}, so the explicit computation of the Green's functions will not be possible in general.
\item \emph{General involutions:} In the theory of differential equations with involutive functions\footnote{Here, for $n\geqslant 2$, we consider a function $f$ to be \emph{involutive order $n$} or an \emph{involution of order $n$} if $f^n=\operatorname{Id}$ and $f^k\ne f$ for $k=2,\dots,n-1$ --cf. \cite{Wie2}. Some other authors consider the term involution only for the case $n=2$, which is standard in other fields, using \emph{finite order operators} for the case presented here.} we have to work with differentiable or at least continuous involutive functions \cite{CabToj} (such as is the case of the reflection), but this poses the severe restriction that continuous involutive functions of order $n$ on connected sets of the real line have to be of order two \cite{McS,cabada2015differential}. This restriction disappears in the context of recurrence relations, which gives rise to three questions worth answering. First, Which are the different involutive functions on ${\mathbb Z}$ for each given order? second, How do the operators which are the pullback of those involutive functions interact with the right shift operator? and last, Under which circumstances can we solve recurrence relations with those involutions?

It is unlikely that we will obtain a full answer to the first question, but we can restrict our research to those involutions that behave well with respect to right shifts. We could start by studying, for instance, involutions that are just transpositions of elements of the sequence since the interaction of the involution with the right shift operator is easily manageable in this case.

More general involutions (that is, involutive operators that are not the pullback by an involutive function) such as $\Lambda$ occurring in Remark~\ref{remlam} are worth studying since they satisfy very attractive properties (for instance, in the case of $\Lambda$, it anticommutes with the right shift operator).

\item \emph{Partial difference equations:} There is also the possibility to move from recurrence in one independent variable to recurrence in several independent variables. Some analogous work has been done previously in the case of partial differential equations with reflection \cite{ToTo}. Again, there is the possibility to study involutions of order greater than two.
 \end{itemize}

\section*{Acknowledgements}
The author would like to acknowledge his gratitude towards Professor Lorena Saavedra for her help regarding the Green's function of linear recurrence relations and the anonymous referees for their helpful comments.

\end{document}